\documentclass[reqno, 11pt]{amsart}

\usepackage[algoruled,lined,noresetcount,norelsize]{algorithm2e}

\usepackage{graphicx}
\usepackage{amsmath,amssymb,amsthm}
\usepackage{mathrsfs}
\usepackage{mathabx}\changenotsign
\usepackage{dsfont}
\usepackage{xcolor}
\usepackage[backref]{hyperref}
\hypersetup{
    colorlinks,
    linkcolor={red!60!black},
    citecolor={green!60!black},
    urlcolor={blue!60!black}
}
\usepackage[T1]{fontenc}
\usepackage{lmodern}
\usepackage[babel]{microtype}
\usepackage[english]{babel}

\usepackage{geometry}
\numberwithin{equation}{section}
\numberwithin{figure}{section}

\usepackage{enumitem}

\theoremstyle{plain}
\newtheorem{thm}{Theorem}[section]

\newtheorem{clm}[thm]{Claim}
\newtheorem{cor}[thm]{Corollary}
\newtheorem{lemma}[thm]{Lemma}
\newtheorem{prob}[thm]{Problem}
\newtheorem{ques}[thm]{Question}

\theoremstyle{definition}

\newtheorem{dfn}[thm]{Definition}

\newcommand\eps{\varepsilon}

\newcommand{\logb}{\log_{b+1}}
\newcommand{\cK}{\mathcal{K}}
\newcommand{\cF}{\mathcal{F}}

\newcommand{\cH}{\mathcal{H}}

\newcommand{\cT}{\mathcal{T}}

\renewcommand{\subset}{\subseteq}
\renewcommand{\setminus}{\smallsetminus}

\DeclareUnicodeCharacter{202A}{}

\begin{document}

\title[Multistage Positional Games]{Multistage Positional Games}
\author[Juri Barkey]{Juri Barkey}
\author[Dennis Clemens]{Dennis Clemens}
\author[Fabian Hamann]{Fabian Hamann}
\author[Mirjana Mikala\v{c}ki]{Mirjana Mikala\v{c}ki}
\author[Amedeo Sgueglia]{Amedeo Sgueglia}

\thanks{
The second author is supported
by Deutsche Forschungsgemeinschaft (Project CL 903/1-1).
The fourth author is partly supported by Ministry of Education, Science and Technological Development, Republic of Serbia (Grant No. ‪451-03-68/2022-14/200125) and Provincial Secretariat for Higher Education and Scientific Research, Province of Vojvodina (Grant No. 142-451-2686/2021).
This research was conducted while the fifth author was a PhD student at the London
School of Economics.}

\address{Hamburg University of Technology, Institute of Mathematics, Am Schwarzenberg-Campus 3, 21073 Hamburg, Germany}
\email{juri.barkey@tuhh.de, dennis.clemens@tuhh.de, fabian.hamann@tuhh.de}
\address{University of Novi Sad, Faculty of Sciences, Department of Mathematics and Informatics, Trg Dositeja Obradovi\'{c}a 4, 21000 Novi Sad, Serbia}
\email{mirjana.mikalacki@dmi.uns.ac.rs}
\address{Department of Mathematics, University College London, Gower Street, London WC1E 6BT, UK}
\email{a.sgueglia@ucl.ac.uk}

\maketitle
%
%
%

\begin{abstract}
We initiate the study of a new variant of the Maker-Breaker positional game, which we call multistage game. Given a hypergraph $\cH=(\mathcal{X},\cF)$ and a bias $b \ge 1$, the $(1:b)$ multistage Maker-Breaker game on $\cH$ is played in several stages as follows. Each stage is played as a usual $(1:b)$ Maker-Breaker game, until all the elements of the board get claimed by one of the players, with the first stage being played on $\cH$. In every subsequent stage, the game is played on the board reduced to the elements that Maker claimed in the previous stage, and with the winning sets reduced to those fully contained in the new board. The game proceeds until no winning sets remain, and the goal of Maker is to prolong the duration of the game for as many stages as possible. In this paper we estimate the maximum duration of the $(1:b)$ multistage Maker-Breaker game, for biases $b$ subpolynomial in $n$, for some standard graph games played on the edge set of $K_n$: the connectivity game, the Hamilton cycle game, the non-$k$-colorability game, the pancyclicity game and the $H$-game. While the first three games exhibit a probabilistic intuition, it turns out that the last two games fail to do so. 

\end{abstract}

%
%
%

\section{Introduction}

Positional games are combinatorial games of perfect information and no chance moves, played by two players who alternately claim previously unclaimed elements of the given board.  Maker-Breaker games are a particular type of positional games, which attracted a lot of attention in the last couple of decades, starting with the seminal papers of~Hales and Jewett~\cite{HJ},  Erd\H{o}s and Selfridge~\cite{ES}, Chv\'{a}tal and Erd\H{o}s~\cite{CE} and Beck~\cite{beck1981positional, beck1982remarks},  to be followed by the groundbreaking results of Stojakovi\'{c} and Szab\'{o}~\cite{milos_tibor}, Gebauer and Szab\'o~\cite{GS09}, and Krivelevich~\cite{K11}. Formally stated,  given any positive integer $b$ and any hypergraph $\mathcal{H}=(\mathcal{X},\mathcal{F})$, a biased $(1:b)$ Maker-Breaker game on $\mathcal{H}$ is defined as follows.
Maker and Breaker alternate in moves, usually with Maker being the first player. In each move, Maker is allowed to claim one previously unclaimed element from the \textit{board} $\mathcal{X}$, and Breaker is allowed to claim up to $b$ such elements. If, throughout the game, Maker manages to occupy all the elements of a \textit{winning set} $F\in\mathcal{F}$, she is declared the winner of the game. Otherwise, i.e.\ if Breaker prevents her from doing so until all elements of $\mathcal{X}$ have been distributed among the two players, Breaker wins the game. Note that these rules imply monotonicity in Breaker's bias~\cite{CE}. In particular, there must be a \textit{threshold bias} $b_{\mathcal{H}}$ such that Breaker wins if and only if $b\geq b_{\mathcal{H}}$. The size of such threshold biases has been investigated for many standard graph games in recent years and we refer the interested reader to~\cite{beck2008combinatorial,hefetz2014positional}.

\smallskip
Maker-Breaker games are played on various boards, but the most studied ones are those played on $E(K_n)$, i.e. the edge set of the complete graph on $n$ vertices. They will also be the focus of this paper, and we will mainly look at the games defined by the hypergraphs
$\mathcal{C}_{n}$,
$\mathcal{HAM}_{n}$,
$\mathcal{H}_{H,n}$,
$\mathcal{PAN}_{n}$ and
$\mathcal{COL}_{n,k}$
on the vertex set $\mathcal{X}=E(K_n)$,
with the hyperedges being the edge sets of
all spanning trees of $K_n$, all Hamilton cycles of $K_n$, all copies of a fixed graph $H$ in $K_n$,
all pancyclic spanning subgraphs of $K_n$ and
all subgraphs of $K_n$ with chromatic number larger than $k$, respectively.

\smallskip

One reason for the strong interest in such games is their deep connection to the theory of random graphs. Indeed, already Chv\'{a}tal and Erd\H{o}s~\cite{CE} suspected the following behaviour for certain biased Maker-Breaker games on $K_n$: the more likely winner in the game between two random players is the same as the winner between two perfect players. 
This behaviour is usually referred to as the \textit{random graph intuition}.
To make it more precise, note that if Maker and Breaker are replaced by ‘random players’
who select their edges uniformly at random, the final subgraph of $K_n$ consisting only of Maker's edges is a random graph $G(n,M)$ chosen uniformly among all graphs on $n$ vertices with $M=\lfloor \frac{1}{b+1}\binom{n}{2} \rfloor$ edges, which is tightly linked to the Erd\H{o}s-Renyi random graph model $G(n,p)$ when $p=\frac{1}{b+1}$.
If the suspicion of Chv\'{a}tal and Erd\H{o}s is true, then the threshold bias $b_{\mathcal{H}}$ should asymptotically be the inverse of the threshold probability ${p}_{\mathcal{H}}$ for the property that $G(n,p)$
contains an element of $\mathcal{H}$. 

\smallskip

For example, it is well known that the threshold probability for connectivity in $G(n,p)$ is $(1+o(1))\frac{\ln n}{n}$, and Gebauer and Szab\'o~\cite{GS09} proved that $b_{\mathcal{C}_n} = (1+o(1)) \frac{n}{\ln n}$, thus showing that the above intuition holds for the connectivity game. This was extended even further by Krivelevich~\cite{K11} to the Hamilton cycle game, by showing that $b_{\mathcal{HAM}_n} = (1+o(1)) \frac{n}{\ln n}$, with the threshold probability for the existence of a Hamilton cycle being $(1+o(1))\frac{\ln n}{n}$, as proved by Koml\'os and Szemer\'edi~\cite{komlos1983limit} and independently by Bollob\'as~\cite{bollobas1984evolution} based on the method of P\'{o}sa~\cite{posa}.
Up to constant factors the same holds for the game $\mathcal{COL}_{n,k}$ with $k$ being a constant, where Hefetz, Krivelevich, Stojakovi\'c, and Szab\'o~\cite{hefetz2008planarity} proved that $b_{\mathcal{COL}_{n,k}}$ is of the order $\frac{n}{k\ln k}$, while the probability threshold for non-$k$-colorability is $(2+o(1))\frac{k \ln k}{n}$.
Further instances of the random graph intuition can be found in e.g.~\cite{beck1996foundations,C16Rand}.

On the other hand, the pancyclicity game and the $H$-game fail to satisfy the above intuition. 
Indeed, Ferber, Krivelevich, and Naves~\cite{F15Gen} proved that the threshold bias $b_{\mathcal{PAN}_n}$ of the pancyclicity game is close to $\sqrt{n}$,
while Cooper and Frieze~\cite{C87Pan} proved that the threshold probability for the pancyclicity property is $(1+o(1))\frac{\ln n}{n}$.
Bednarska and \L uczak~\cite{bednarska2000biased} proved that the threshold bias $b_{\mathcal{H}_{H,n}}$ of the $H$-game is of the order $n^{1/m_2(H)}$
if $H$ has at least two edges, where $m_2(H)=\max_{G\subseteq H \atop v(G)\geq 3} \frac{e(G)-1}{v(G)-2}$ denotes the \textit{maximum 2-density} of $H$. However it is well known~\cite{bollobas1981threshold} that the threshold probability for the random graph $G(n,p)$ to contain a copy of $H$ is $n^{-1/m(H)}$, where $m(H)=\max_{G\subseteq H \atop v(G)\geq 1} \frac{e(G)}{v(G)}$ denotes the \textit{maximum density} of $H$.

Although these two games provide examples for which the random graph intuition is not true, it should be noted that a deep connection to random graphs still exists. In fact, in both cases it turns out that the results on the threshold biases are linked to \textit{resilience properties} of $G(n,p)$. 

\smallskip

In this paper, we initiate a different perspective on the random graph intuition. While in the $(1:b)$ Maker-Breaker game, Maker's graph is forced to be sparse when Breaker's bias is large, we now want to force sparseness by playing multiple games consecutively, where each new game shrinks the board on which the next game is allowed to be played.

Given a hypergraph $\cH=(\mathcal{X},\cF)$ and a bias $b \ge 1$, we define the \textit{multistage $(1:b)$ Maker-Breaker game} on $\mathcal{H}$ as follows. The game proceeds 
in several stages, with each stage being played as a usual $(1:b)$ Maker-Breaker game.
For convenience we define $\mathcal{X}_0:=\mathcal{X}$, $\cF_0:=\cF$ and $\cH_0:=\cH$.
Then, for $i \ge 1$, in Stage $i$, Maker and Breaker play on the board $\mathcal{X}_{i-1}$, consider the hypergraph $\mathcal{H}_{i-1}=(\mathcal{X}_{i-1},\mathcal{F}_{i-1})$, and alternate in turns in which Maker occupies exactly one unclaimed element of $\mathcal{X}_{i-1}$, and afterwards Breaker occupies up to $b$ unclaimed elements of $\mathcal{X}_{i-1}$, with Maker moving first.
Once all the elements of $\mathcal{X}_{i-1}$ have been distributed among both players, we let $\mathcal{X}_i\subset \mathcal{X}_{i-1}$ be the set of all elements claimed by Maker in Stage $i$, and we let $\mathcal{F}_i=\left\{F\in\mathcal{F}_{i-1}:~ F\subset \mathcal{X}_i\right\}$ be the set of all remaining winnings sets (from $\mathcal{F}$) that Maker managed to fully occupy in Stage $i$.
Observe that this defines a new hypergraph $\cH_i:=(\mathcal{X}_i,\cF_i)$.
We stop the game the first time that there are no winning sets left anymore, and our main question is how long Maker can delay the stop of a given game.
For that, we define the threshold parameter $\tau(\mathcal{H},b)$ as the largest number $s$ such that, in the $(1:b)$ multistage Maker-Breaker game on hypergraph $\mathcal{H}$, Maker has a strategy to ensure $\mathcal{F}_s\neq \varnothing$ and thus to play at least $s$ stages.

\medskip

A random graph intuition in this setting corresponds to the assumption that a random game is likely to last as long as a perfectly played game.
Note that if a multistage $(1:b)$ Maker-Breaker game on $K_n$ is played by two random players, then $\mathcal{X}_i$ is the edge set of a uniform random graph
$G(n,M)$ with roughly $\left( \frac{1}{b+1} \right)^i \binom{n}{2}$ edges.
Therefore, if true, the random graph intuition would suggest that Maker can maintain a connected spanning subgraph, a Hamilton cycle, or a non-$k$-colorable graph for roughly $\log_{b+1}(n)-\log_{b+1}(\ln (n))$ stages. We show this is indeed asymptotically the best that Maker can do.  The following statements assume that $b$ is \textit{subpolynomial} in the number of vertices $n$, which means that $b = o(n^\eps)$ for all $\eps > 0$.

\begin{thm}[Multistage Hamilton cycle game]\label{thm:HAM.multi}
If $b$ is subpolynomial in $n$, then
$$
\tau(\mathcal{HAM}_{n},b) = ( 1 + o(1) ) \log_{b+1}(n)\, .
$$
\end{thm}

\begin{cor}[Multistage connectivity game]\label{cor:Con.multi}
If $b$ is subpolynomial in $n$, then
$$
\tau(\mathcal{C}_{n},b) = ( 1 + o(1) ) \log_{b+1}(n)\, .
$$
\end{cor}

\begin{thm}[Multistage non-$k$-colorability game]\label{thm:kCol.multi}
If $b$ and $k$ are subpolynomial in $n$, and $k \geq 2$, then
$$
\tau(\mathcal{COL}_{n,k},b) = \left( 1 + o(1) \right) \log_{b+1}(n)\, .
$$
\end{thm}

However the random graph intuition fails for both the multistage $H$-game and the multistage pancyclicity game, as, if played  randomly, they would typically last $\left( \frac{1}{m(H)} + o(1) \right) \log_{b+1}(n)$ and $(1+o(1)) \logb(n)$ stages, respectively, while we can show the following.

\begin{thm}[Multistage $H$-game]\label{thm:H-game.multi}
If $b$ is subpolynomial in $n$, then
$$
\tau(\mathcal{H}_{H,n},b) = \left( \frac{1}{m_2(H)} + o(1) \right) \log_{b+1}(n)\, .
$$
\end{thm}

\begin{thm}[Multistage pancyclicity game]\label{thm:PAN.multi}
If $b$ is subpolynomial in $n$, then
$$
\tau(\mathcal{PAN}_{n},b) = \left( \frac{1}{2} + o(1) \right) \log_{b+1}(n)\, .
$$
\end{thm}

Observe that in all our results, it happens that the threshold $\tau(\mathcal{H},b)$ is asymptotically the same as
$\log_{b+1}(b_{\mathcal{H}})$, where $b_{\mathcal{H}}$ denotes the threshold bias discussed earlier. We believe that this is not a coincidence.

\smallskip

Using Lehman's Theorem~\cite{lehman}, we can even provide a precise result for the $(1:1)$ multistage connectivity game, showing that Maker can do slightly better than what the above random graph argument suggests.

\begin{thm}[Unbiased connectivity game]\label{thm:Conn.multi_unbiased}
$$
\tau(\mathcal{C}_{n},1) = \lfloor \log_2(n) - 1 \rfloor\, .
$$
\end{thm}

\smallskip

For most of our Maker strategies, we will prove and apply a winning criterion (Lemma~\ref{lem:multistagediscrepancy}) that enables Maker to occupy at least one element of each member of a suitable collection of subsets of the initial board for several stages. Its proof is based on the potential function method, and we believe that the statement itself and the steps towards its proof (Lemma~\ref{lem:alphafraction} and Lemma~\ref{lem:biaseddiscrepany}) may be of independent interest.

\medskip

%
%

\subsection{Notation.} We use standard graph theory notation throughout. 
For a given graph $G$, we denote its vertex set and edge set by $V(G)$ and $E(G)$, respectively, and we set $v(G) = |V(G)|$ and $e(G) = |E(G)|$.\\
Given two vertices, $x$ and $y$, an edge is denoted by $xy$. If an edge is unclaimed by any of the players we call it \textit{free}. Given a subset $A\subseteq V(G)$, by $G[A]$ we denote the \textit{induced} subgraph of $G$ with vertex set $A$, and we further set $E_G(A)=E(G[A])$. For disjoint sets $A,B \subseteq V(G)$, we let $E_G(A,B)$ denote the set of edges of $G$ with one endpoint in $A$ and one endpoint in $B$. Given graphs $H, F$, we write $H\subseteq F$ to denote that $H$ is contained in $F$, meaning that $V(H)\subseteq V(F)$ and $E(H)\subseteq E(F)$. 
Given any set $S\subseteq V(G)$, we denote the (joint) external neighbourhood of $S$ by $N_G(S)=\{u \in V(G)\setminus S : ux \in E(G), x\in S\}$.
The density of a graph $G$ is defined as $d(G)=\frac{e(G)}{v(G)}$, while its \textit{maximum density} is $m(G)=\max_{H\subseteq G \atop v(H) \geq 1} d(H)$. If $G$ has at least three vertices, we define its $2$-density as $d_2(G)=\frac{e(G)-1}{v(G)-2}$, and its \textit{maximum 2-density} as $m_2(G)=\max_{H\subseteq G \atop v(H)\geq 3} d_2(H)$.\\
We denote the natural logarithm of $n$ by $\ln(n)$. We say that a function $f(n)$ is \textit{subpolynomial} in $n$ if $f(n) = o(n^\eps)$ for all $\eps > 0$.

\medskip

\subsection{Organization of the paper.}
The rest of the paper is organized as follows. In Section~\ref{sec:criteria} we collect and prove some winning criteria that will be used in the proofs of most of our theorems. Then in Section~\ref{sec:HAM.multi} we consider the Hamilton cycle game and prove Theorem~\ref{thm:HAM.multi} and Corollary~\ref{cor:Con.multi}, in Section~\ref{sec:kCol.multi} we discuss the game $\mathcal{COL}_{n,k}$ and prove Theorem~\ref{thm:kCol.multi}, in Section~\ref{sec:H-game.multi} we study the $H$-game and prove Theorem~\ref{thm:H-game.multi}, in Section~\ref{sec:PAN.multi} we discuss the pancyclicity game and prove Theorem~\ref{thm:PAN.multi}, and in 
Section~\ref{sec:Conn.multi_unbiased} we discuss the unbiased connectivity game and prove Theorem~\ref{thm:Conn.multi_unbiased}.
Finally, in Section~\ref{sec:concluding}
we provide some concluding remarks and pose several open problems.

\bigskip

%
%
%

\section{Winning Criteria}\label{sec:criteria}

This section deals with the two winning criteria (Theorem~\ref{beck_variant} and Lemma~\ref{lem:multistagediscrepancy}) that we will often use in the proofs of Breaker's and Maker's strategies, respectively. 
The following theorem shows that, under a certain Breaker's strategy, we can give an upper bound on the number of winning sets that Maker can completely occupy during a Maker-Breaker game.
This is a variant of Beck's Criterion~\cite{beck85Rand} and its proof, which can be found in \cite{bednarska2000biased}, follows directly from the proof of that criterion.

\begin{thm}[Lemma 5 in~\cite{bednarska2000biased}]\label{beck_variant}
Let an integer $b \ge 1$ and a hypergraph $\mathcal{H}=(\mathcal{X},\mathcal{F})$ be given.
Then, in the $(1:b)$ Maker-Breaker game on $\mathcal{H}$, Breaker has a strategy which ensures that Maker occupies no more than
$$
\sum_{F\in\mathcal{F}} (1+b)^{-|F|+1} 
$$
winning sets $F\in\mathcal{F}$ completely.
\end{thm}

We now discuss Maker's strategy.
It might not be clear at this point, but it will often be the case that, while playing a $(1:b)$ multistage Maker-Breaker game on $K_n$, Maker wants to claim at least one edge in each member of a suitable family of edge sets of $K_n$.
We provide a criterion under which Maker can achieve such a goal in Lemma~\ref{lem:multistagediscrepancy}. 
The methods used in the proof of Lemma~\ref{lem:multistagediscrepancy} are similar to those in Chapters 17 and 20 of~\cite{beck2008combinatorial}.
However the results in~\cite{beck2008combinatorial} only deal with uniform hypergraphs and a single stage game, and thus are not applicable in our setting, as we work with non-uniform hypergraphs and multistage games.

As a first step towards proving Lemma~\ref{lem:multistagediscrepancy}, we need to generalise a criterion of Beck~\cite{beck1981van}, which ensures that Maker can get an $\alpha$-fraction of each winning set in a biased Maker-Breaker game.
We believe that this criterion could be of independent interest.

\begin{lemma}\label{lem:alphafraction}
Given a hypergraph $\mathcal{H} = (\mathcal{X}, \mathcal{F})$, a real $0 \leq \alpha \leq 1$, and an integer $b \geq 1$,
if there exists $\mu \in (0,1)$ s.t.
\[\sum_{F \in \mathcal{F}} \lambda_{\alpha,\mu,b}^{-|F|} < 1 \text{ with } \lambda_{\alpha,\mu,b} := (1+\mu)^{\frac{1-\alpha}{b}}(1-\mu)^{\alpha},\] 
then Maker has a strategy to claim $\alpha |F|$ elements of every winning set $F \in \mathcal{F}$ in a $(1:b)$ Maker-Breaker game.
\end{lemma}

\begin{proof}
Denote by $X_i$ the elements which Maker took in the first $i$ rounds. 
Denote by $Y_{i,j}$ with $0 \leq j \leq b$ the elements which Breaker took in the first $i-1$ rounds plus the first $j$ elements he took in the $i^\text{th}$ round. 
We define the following potential for each $F \in \mathcal{F}$, $0 \leq j \leq b$ and round $i$:
    \[\phi_{i,j}(F) := (1+\mu)^{\frac{1}{b}(|F \cap Y_{i,j}| - (1-\alpha)|F|)}(1-\mu)^{|F \cap X_i| - \alpha|F|}.\] 
We further define the potential of a vertex $v \in \mathcal{X}$ as \[\phi_{i,j}(v) := \sum_{\substack{F \in \mathcal{F} \\ v \in F}} \phi_{i,j}(F) \, ,\]
and we let 
\[\phi_{i,j} := \sum_{F \in \mathcal{F}} \phi_{i,j}(F)\] 
denote the total potential of the game immediately after Breaker took his
$j^\text{th}$ element of round $i$.
Note that $\phi_{i,0}(v)$ and $\phi_{i,0}$ then describe potentials immediately after Maker's $i^\text{th}$ turn. 
We further let \[\phi_{0,b} := \sum_{F \in \mathcal{F}} (1+\mu)^{-\frac{1}{b}(1-\alpha)|F|}(1-\mu)^{ - \alpha|F|} = \sum_{F \in \mathcal{F}} \lambda_{\alpha,\mu,b}^{-|F|}\] 
denote the potential at the start of the game, i.e.~when no elements have been claimed yet. 
By assumption $\phi_{0,b} < 1$.

Maker's strategy in round $i$ is to claim an element $v \in \mathcal{X}$ not yet claimed which maximises $\phi_{i-1,b}(v)$. 
We claim that, following this strategy, Maker can ensure that $\phi_{i,j}<1$ holds throughout the game.
Indeed, for $j \in [b]$, let $w_j$ be the $j^\text{th}$ element Breaker claimed in round $i$. 
Then, $\phi_{i-1,b}(v) \geq \phi_{i-1,b}(w_j) \geq \phi_{i,0}(w_j)$ 
where the first inequality holds by the maximality of $v$, and the second inequality holds since Maker's move can never increase the potential of any vertex.
Moreover, $\phi_{i,j}(w) \leq (1+\mu)^{1/b}\phi_{i,j-1}(w)$ for all $w \in \mathcal{X}$, as Breaker can increase the potential of any $F\in \mathcal{F}$
at most by a factor $(1+\mu)^{1/b}$ when he claims an element of $\mathcal{X}$.
In particular, this gives $\phi_{i,j-1}(w_j)\leq (1+\mu)^{(j-1)/b}\phi_{i,0}(w_j) \le (1+\mu)^{(j-1)/b}\phi_{i-1,b}(v)$ for all $j \in [b]$. Therefore, 
 \begin{align*}
        \phi_{i,b} &= \phi_{i-1,b} - \mu \phi_{i-1,b}(v) + ((1+\mu)^{1/b} - 1) \sum_{j=1}^b \phi_{i,j-1}(w_j) \\
                    &\leq \phi_{i-1,b} - \mu \phi_{i-1,b}(v) + ((1+\mu)^{1/b} - 1)\phi_{i-1,b}(v) \sum_{j=1}^b (1+\mu)^{(j-1)/b} \\
                    &= \phi_{i-1,b} - \mu \phi_{i-1,b}(v) + \mu \phi_{i-1,b}(v) = \phi_{i-1,b},
    \end{align*}
where the first line uses that $\mu \phi_{i-1,b}(v)$ describes the change of the total potential caused by Maker and $((1+\mu)^{1/b} - 1) \sum_{j=1}^b \phi_{i,j-1}(w_j)$ is the change caused by Breaker, and where we use the geometric sum to get the third line. 
Further, we have that $\phi_{i,j} \leq \phi_{i,b}$ for all $0 \leq j \leq b$, because the potential can only increase when Breaker claims an element. 
We thus conclude that $\phi_{i,j} \leq \phi_{i,b} \leq \phi_{0,b} < 1$ for all $i$ and $0 \leq j \leq b$. 
    
Now assume that there is some $F \in \mathcal{F}$ such that Breaker has claimed at least $(1-\alpha) |F|$ elements of $F$ after some round $i$. Note that this implies $\phi_{i,b}(F) \geq 1$, and therefore $\phi_{i,b} \geq 1$, which is a contradiction.
\end{proof}

Given any fixed winning set $F\in\mathcal{F}$ and any Breaker's bias $b$, it is clear that Maker can guarantee to claim roughly a $\frac{1}{b+1}$-fraction of all elements of $F$. Our next aim towards proving Lemma~\ref{lem:multistagediscrepancy} is to show that,
under certain conditions, Maker can simultaneously ensure to get almost a $\frac{1}{b+1}$-fraction 
from each winning set $F\in\mathcal{F}$. The following lemma is obtained from Lemma~\ref{lem:alphafraction}
by a suitable choice of the parameters $\alpha$ and $\mu$. Note that we allow the family $\cF$ of winning sets to be split into $s$ subfamilies $\cF_1,\dots,\cF_s$, as we often need to deal with winning sets of different kinds, for example while trying to create an expander in the Hamilton cycle game.

\begin{lemma}\label{lem:biaseddiscrepany}
    Let $\mathcal{X}$ be a set of size $n$ and let $\delta = \delta(n) \in (0,1)$. Let $b,s \geq 1$ be integers, $\mathcal{H} = (\mathcal{X}, \mathcal{F} = \mathcal{F}_1 \cup \dots \cup \mathcal{F}_s)$ a hypergraph, and $k_i = \min_{F \in \mathcal{F}_i} |F|$ for $i \in [s]$. If $k_i > 4\delta^{-2}\ln(s|\mathcal{F}_i|)$ for every $i \in [s]$, then Maker has a strategy to claim at least $(\frac{1}{b+1} - \delta)|F|$ elements of every winning set $F \in \mathcal{F}$ in a $(1:b)$ Maker-Breaker game. 
\end{lemma}

\begin{proof}
    With $\mu := \frac{1}{2}\delta$, let $\eps > 0$ such that 
    \begin{equation}\label{eq:bd1}
        e^{\mu^2} 
        = (1+\mu)^{\frac{1}{b+1} + \frac{\eps}{b}}(1-\mu)^{\frac{1}{b+1} - \eps}.
    \end{equation}
	Note that the existence of such $\eps$ is given by the fact that 
	$f(x)=   (1+\mu)^{\frac{1}{b+1} + \frac{x}{b}}(1-\mu)^{\frac{1}{b+1} - x}$
	defines a continuous function on $\mathbb{R}$ with $f(0)=(1-\mu^2)^{\frac{1}{b+1}}<1$ and 
	$\lim_{x\rightarrow +\infty} f(x)=+\infty$. 
    Moreover, rearranging \eqref{eq:bd1} gives
    \begin{align*}
            \eps = \frac{\mu^2 - \frac{1}{b+1}\ln(1-\mu^2)}{\frac{1}{b}\ln(1+\mu) - \ln(1-\mu)}\, .
    \end{align*}
	Therefore, using that    
    $x - x^2 < \ln(1+x) < x$
    holds for all $x \neq 0, |x| \leq \frac{1}{2}$,
    we obtain
    \begin{align*}
    \eps  < \frac{\mu^2 + \frac{1}{b+1}(\mu^4 + \mu^2)}{\frac{1}{b}(\mu-\mu^2) + \mu}
          <\left(1 + \frac{2}{b+1}\right)\cdot\frac{\mu^2}{\mu}
          \leq 2\mu.
    \end{align*}
    We now apply Lemma~\ref{lem:alphafraction} with $\alpha = \frac{1}{b+1} - \eps$. Note that $\alpha > 0$ by the choice of $\eps$ and $\mu$ unless $\frac{1}{b+1} - \delta \leq 0$, but in that case, there is nothing to prove.
    By (\ref{eq:bd1}) and the choice of $\mu$, we get 
 \[\sum_{F \in \mathcal{F}} ((1+\mu)^{\frac{1}{b+1} + \frac{\eps}{b}}(1-\mu)^{\frac{1}{b+1} - \eps})^{-|F|} \leq \sum_{i=1}^s\sum_{F \in \mathcal{F}_i}e^{-\mu^2|F|} \leq \sum_{i=1}^s |\mathcal{F}_i| e^{-\frac{1}{4}\delta^2 k_i} < \sum_{i=1}^s|\mathcal{F}_i|e^{-\ln(s|\mathcal{F}_i|)} = \frac{s}{s} = 1.\]
    Thus, Maker can claim at least 
    \[\alpha|F| = \left(\frac{1}{b+1} - \eps\right)|F| > \left(\frac{1}{b+1} - 2\mu\right)|F| = \left(\frac{1}{b+1} - \delta\right)|F|\] 
    elements of every winning set $F \in \mathcal{F}$.
\end{proof}

Finally, we may apply Lemma~\ref{lem:biaseddiscrepany} repeatedly over several stages in order to obtain that,
even after some number of stages, Maker can make sure to get at least one element from every winning set.

\begin{lemma}\label{lem:multistagediscrepancy}
    Let $\mathcal{X}$ be a set of size $n$ and let $\gamma = \gamma(n) \in (0,1)$. Given integers $b,s \geq 1$, and a hypergraph 
    $\mathcal{H} = (\mathcal{X}, \mathcal{F}= \mathcal{F}_1 \cup \dots \cup \mathcal{F}_s)$ with $|\cF|>1$ 
    the following holds. 
    Let $k_j := \min_{F \in \mathcal{F}_j} |F|$ for every $j \in [s]$ and assume that
    \begin{equation}    \label{property:b} 
    \left( \frac{k_j}{\ln(s|\mathcal{F}_j|)} \right)^{\gamma/2}  \geq 20b
    \cdot \max\left\{1, \logb\left( \frac{k_j}{\ln(s|\mathcal{F}_j|)} \right) \right\}
    \end{equation}    
    for every $j\in [s]$.
    Then, in the $(1:b)$ Maker-Breaker multistage game on $\mathcal{H}$, Maker has a strategy to ensure that
    after $(1-\gamma)\min_{i \in [s]}\logb\left(\frac{k_i}{\ln(s|\mathcal{F}_i|)}\right)$ stages, she still claims at least one element in each $F\in\mathcal{F}$.
\end{lemma}

\begin{proof}
Set $t:= (1-\gamma)\min_{i \in [s]}\logb\left(\frac{k_i}{\ln(s|\mathcal{F}_i|)}\right)$ and 
$\delta := 4\max_{i \in [s]} \left(\frac{\ln(s|\mathcal{F}_i|)}{k_i}\right)^{\frac{\gamma}{2}}$, and observe that there exists some $j\in [s]$ with
$$
2b\delta t 
\leq
2b\cdot  4 \left(\frac{\ln(s|\mathcal{F}_j|)}{k_j}\right)^{\frac{\gamma}{2}} \cdot 
(1-\gamma) \logb\left(\frac{k_j}{\ln(s|\mathcal{F}_j|)}\right) \stackrel{~\eqref{property:b}}{<} \frac{1}{2}.
$$
At the end of Stage $i$, let $\mathcal{X}^i$ denote the set of all elements of $\mathcal{X}$ that belong to Maker, let $\mathcal{F}_j^i:=\left\{F\cap \mathcal{X}^i:~ F\in \mathcal{F}_j \right\}$ be the multiset of "left-overs" of the winning sets of $\mathcal{F}_j$, and set $k_{i,j} := \min_{F \in \mathcal{F}_j^i} |F|$ for each $j \in [s]$.  We aim to show that Maker can play in such a way that $k_{t,j}\geq 1$ for all $j\in [s]$.
To achieve that, in each stage $i\leq t$, we let Maker play according to the strategy of Lemma~\ref{lem:biaseddiscrepany} with input $\delta$ and hypergraph $\mathcal{H}^i:=(\mathcal{X}^i,\mathcal{F}^i :=\mathcal{F}_1^i\cup\ldots\cup\mathcal{F}_s^i)$.

Notice that, as long as the assumptions of Lemma~\ref{lem:biaseddiscrepany} hold at the beginning of some stage $i\leq t$, i.e. if $k_{i-1,j} > 4\delta^{-2} \ln(s|\mathcal{F}_j|)$ for all $j\in [s]$,
we obtain $k_{i,j} \geq k_{i-1,j}(\frac{1}{b+1} - \delta)$ and hence
$k_{i,j} \geq k_j(\frac{1}{b+1} - \delta)^i$ for every $j \in [s]$.
In particular, we then conclude
\begin{align*}
k_{i,j} & \geq k_j\left(\frac{1}{b+1} - \delta\right)^t 
	   = k_j \left( \frac{1}{b+1} \right)^t \cdot (1-(b+1)\delta)^t \\
	& \geq \left( \frac{k_j}{\ln(s|\mathcal{F}_j|)} \right)^{\gamma} \ln(s|\mathcal{F}_j|) \cdot (1-2b\delta t) \geq \frac{1}{2} \left( \frac{k_j}{\ln(s|\mathcal{F}_j|)} \right)^{\gamma} \ln(s|\mathcal{F}_j|) \, ,
\end{align*}
where the second inequality holds since $t\leq (1-\gamma) \logb\left(\frac{k_j}{\ln(s|\mathcal{F}_j|)}\right) $ and since $(1-x)^t\geq 1-xt$ as long as $x<1$, and the last inequality holds since $2b\delta t < \frac{1}{2}$.
This in turn gives that
\begin{align*}
k_{i,j}  \geq \frac{1}{2} \left( \frac{k_j}{\ln(s|\mathcal{F}_j|)} \right)^{\gamma} \ln(s|\mathcal{F}_j|)  
	 \geq 4\delta^{-2} \ln(s|\mathcal{F}_j|)
	= 4\delta^{-2} \ln(s|\mathcal{F}^i_j|) \, , 
\end{align*} 
where the second inequality holds by the definition of $\delta$. Now, this means that for the next stage $i+1$ we can again apply Lemma~\ref{lem:biaseddiscrepany}. Inductively it follows that Maker can play in such a way that 
$$
k_{i,j}  \geq \frac{1}{2} \left( \frac{k_j}{\ln(s|\mathcal{F}_j|)} \right)^{\gamma} \ln(s|\mathcal{F}_j|) 
$$
for all $i\leq t$ and $j\in [s]$. In particular, using~\eqref{property:b} again, we obtain
$k_{t,j} \geq 1$.
\end{proof}

\bigskip

%
%
%

\section{Hamilton cycle game}\label{sec:HAM.multi}

In this section we prove
Theorem~\ref{thm:HAM.multi}
and Corollary~\ref{cor:Con.multi}.
The upper bounds $\tau(\mathcal{C}_n,b),\tau(\mathcal{H}_n,b)\leq (1+o(1))\logb(n)$ are trivial, since after $(1+o(1))\logb(n)$ stages, the board has fewer than $n-1$ edges, and thus it can be neither Hamiltonian, nor connected. 
Thus it remains to provide a strategy for Maker to obtain a matching lower bound for the Hamilton cycle game, which immediately gives a lower bound for the connectivity game as well.  
We will use the following criterion for the existence of Hamilton cycles,
which is obtained by choosing $d=\ln\ln n$ in Theorem 1.1 in \cite{hamilton_expansion}.
We recall that, given a set $S \subseteq V(G)$, the notation $N_G(S)$ stands for the external neighbourhood of $S$.

\begin{thm}[Corollary from Theorem 1.1 in \cite{hamilton_expansion}]
\label{thm:hamilton_expansion}
For every large enough $n$ the following holds.
Let $G=(V,E)$ be a graph on $n$ vertices such that 
    \begin{enumerate}[label=\upshape(P\arabic*)]
        \item \label{hamilton-expansion} for every $S \subseteq V$, if $|S| \le \frac{n}{\ln n}$, then $|N_G(S)| \ge (\ln\ln n) |S|$;
        \item \label{hamilton-connectivity} there is an edge in $G$ between any two disjoint subsets $A, B  \subseteq V$ such that $|A|=|B| = \frac{n}{\ln n}$,
    \end{enumerate}
    then $G$ is Hamiltonian.
\end{thm}

\subsection{Proof of the lower bound in Theorem~\ref{thm:HAM.multi} (Maker's strategy)} 

Choose
$$\eps\in \left( \frac{4\ln\ln n}{\ln n}, \frac{5\ln\ln n}{\ln n} \right)
  ~~ \text{and} ~~ 
 \gamma := 2\cdot \frac{\ln(b) + \ln\logb(n)+5}{(1-\eps)\ln(n) -2\ln\ln(n) - \ln(2)}$$
such that 
$\eps^{-1}\in \mathbb{N}$, and assume $n$ to be large enough whenever needed.
We will prove that
$\tau(\mathcal{HAM}_n,b)\geq (1-\gamma-2\eps) \log_{b+1}(n)$.
Since $b$ is subpolynomial, then 
$\tau(\mathcal{HAM}_n,b)\geq (1-o(1)) \log_{b+1}(n)$, as required.
Moreover observe that when $b=1$, this gives
$\tau(\mathcal{HAM}_n,b)\geq \log_{2}(n)- \Theta(\log\log (n))$
which is very close to the random graph intuition as discussed in the introduction.

\medskip

In order to prove this bound, consider the family $\mathcal{F}:=\bigcup_{i=1}^s \mathcal{F}_i$ with $s:=\frac{2}{\eps}$ and
\begin{align*}
	\mathcal{F}_j & := \left\{E_{K_n}(A,B):~A,B \subseteq V(K_n), A\cap B=\varnothing,~ |A|=n^{\frac{(j-1)\eps}{2}},~ |B|=n-\frac{1}{2}n^{\frac{(j+1)\eps}{2}} \right\} ~ \text{ for }1\leq j \leq s-1, \\
	\mathcal{F}_s & := \left\{E_{K_n}(A,B):~A,B \subseteq V(K_n), A\cap B=\varnothing,~ |A|=|B|=\frac{n}{\ln n} \right\}\, , 
\end{align*}
and set $k_j:=\min\{|F|:~F\in\mathcal{F}_j\}$ for every $j\in [s]$. 

\medskip

Maker plays according to Lemma~\ref{lem:multistagediscrepancy} with $\gamma$ as defined above and hypergraph $\cH:=(E(K_n), \cF)$.
For that we need to verify that the condition~\eqref{property:b} holds.
For $j\in [s-1]$, we have
$\frac{1}{2}n^{1+\frac{(j-1)\eps}{2}} \leq k_j \leq n^{1+\frac{(j-1)\eps}{2}}$ and
$|\mathcal{F}_j| = \displaystyle \binom{n}{\frac{1}{2}n^{(j+1)\eps/2}}\binom{\frac{1}{2}n^{(j+1)\eps/2}}{n^{(j-1)\eps/2}}$.
Hence
$$|\mathcal{F}_j|\leq n^{n^{(j+1)\eps/2}} ~~ \text{and} ~~
|\mathcal{F}_j|\geq \binom{n}{\frac{1}{2}n^{(j+1)\eps/2}}
\geq \left(2n^{1-(j+1)\eps/2}\right)^{\frac{1}{2}n^{(j+1)\eps/2}}
\geq 2^{\frac{1}{2}n^{(j+1)\eps/2}}.$$ 
In particular, $
\frac{n^{1-\eps}}{2\ln^2 n} \leq \frac{k_j}{\ln (s|\mathcal{F}_j|)} \leq n,
$ 
from which we can conclude~\eqref{property:b} for $j\in [s-1]$ as follows:
\begin{align*}
\left( \frac{k_j}{\ln(s|\mathcal{F}_j|)} \right)^{\gamma/2} 
 & \geq \left( \frac{n^{1-\eps}}{2\ln^2n} \right)^{\gamma/2} 
  = \exp\left(\ln(b)+\ln\logb(n)+5\right) \\
  & > 20b \logb(n)
 \geq 20b \max\left\{1,\logb\left( \frac{k_j}{\ln(s|\mathcal{F}_j|)} \right)\right\}\, .
\end{align*}
Moreover, for $j=s$ we have $k_s=\frac{n^2}{\ln^2 n}$ and $|\mathcal{F}_s|=\frac{1}{2}\binom{n}{n/\ln n}\binom{n-n/\ln n}{n/\ln n}$, hence
$|\mathcal{F}_s|\leq 4^n$ and $|\mathcal{F}_s|\geq \exp \left(\frac{(2-o(1))n \ln\ln n}{\ln n}\right)$. In particular,
$
\frac{n}{2\ln^2n} \leq \frac{k_s}{\ln (s|\mathcal{F}_s|)} \leq n
$
from which we can conclude~\eqref{property:b} analogously.

\medskip

By Lemma~\ref{lem:multistagediscrepancy} we now obtain that Maker can ensure to claim at least one edge in each $F\in \mathcal{F}$ for at least
$$
(1-\gamma) \logb\left(\frac{n^{1-\eps}}{2\ln^2 n}\right)
\geq 
(1-\gamma) \logb\left(n^{1-2\eps}\right)
\geq
(1-\gamma-2\eps) \logb (n)
$$
stages, where we use the choice of $\eps$ in the first inequality.

\bigskip

It thus remains to check that if a graph $G$ contains an edge from every $F \in \cF$, then properties~\ref{hamilton-expansion} and \ref{hamilton-connectivity}
hold. Using only $\mathcal{F}_s$, we immediately see that \ref{hamilton-connectivity} holds.
For proving~\ref{hamilton-expansion}, 
let $S\subseteq V(K_n)$ with $|S|\leq \frac{n}{\ln n}$ be given,
and let $j\in [s-1]$ be largest such that $n^{(j-1)\eps/2}\leq |S|$. 
If $j\leq s-2$, then $|S|<n^{j\eps/2}$. Otherwise, if $j=s-1$, then $|S|\leq \frac{n}{\ln n}=\frac{n^{(j+1)\eps/2}}{\ln n}$. In any case, choose any subset $S'\subseteq S$
of size $n^{(j-1)\eps/2}$, and observe that $|N_{G}(S')|\geq \frac{1}{2}n^{(j+1)\eps/2}-n^{(j-1)\eps/2}$ since Maker claims an edge in every set $E_{K_n}(S',B)\in \mathcal{F}_j$ with $B\cap S'=\varnothing$ and $|B|=n-\frac{1}{2}n^{(j+1)\eps/2}$. In particular,
\begin{align*}
|N_{G}(S)|\geq |N_{G}(S')|-|N_{G}(S')\cap S| & \geq 
\frac{1}{2}n^{(j+1)\eps/2} -n^{(j-1)\eps/2} - \frac{n^{(j+1)\eps/2}}{\ln n} \\
& \geq \frac{1}{4}n^{(j+1)\eps/2}
> (\ln\ln n)|S|\, ,
\end{align*}
where we use the definition of $j$ and that $\eps>\frac{4\ln\ln n}{\ln n}$.
This proves~\ref{hamilton-expansion} and hence
the lower bound of Theorem~\ref{thm:HAM.multi}. \hfill $\Box$

\bigskip

%
%
%

\section{Non-k-colorability game}\label{sec:kCol.multi}

In this section we prove Theorem~\ref{thm:kCol.multi} and we split its proof into Maker's and Breaker's part.

\subsection{Proof of the lower bound in Theorem~\ref{thm:kCol.multi} (Maker's strategy)}

Let $n$ be large enough and consider the family
\[\mathcal{F} := \left\{E_{K_n}(A) : A \subseteq V(K_n), |A| = \left\lceil\frac{n}{k}\right\rceil\right\}.\] 
Maker plays according to Lemma~\ref{lem:multistagediscrepancy} with $s=1$,
\[\gamma := 2 \cdot \frac{\ln(b) + \ln \logb(n) + 5}{\ln(n) - 2\ln(k) - \ln(4) - \ln \ln(2)},\] and hypergraph $\mathcal{H}:=(E(K_n), \cF)$. Note that $\gamma = o(1)$ and $\gamma > 0$ since $b$ and $k$ are subpolynomial in $n$. In order to apply Lemma~\ref{lem:multistagediscrepancy}, we need to check that the condition~\eqref{property:b} holds.
Note that $|F| = \binom{\left\lceil\frac{n}{k}\right\rceil}{2}$ for every $F \in \mathcal{F}$ and, with $\ell:=\binom{\left\lceil\frac{n}{k}\right\rceil}{2}$, we have
\begin{align*}
    \left(\frac{\ell}{\ln(|\mathcal{F}|)}\right)^{\frac{\gamma}{2}} \geq \left(\frac{\frac{n^2}{4k^2}}{\ln(2^n)}\right)^{\frac{\gamma}{2}} & = \exp\left( \ln(b) + \ln(\logb(n)) + 5\right)\\
    & > 20b\logb(n) \geq 20b\max\left\{1,\logb\left(\frac{\ell}{\ln(|\mathcal{F}|)}\right)\right\},
\end{align*}
where we use $\frac{\ell}{\ln(|\cF|)} \le n$ in the last inequality.
Therefore,~\eqref{property:b} holds, and we conclude that Maker can claim an edge from each $F \in \mathcal{F}$ for at least 
\begin{align*}
    (1 - \gamma)\logb\left(\frac{\ell}{\ln(|\mathcal{F}|)}\right) &\geq (1 - \gamma)\logb\left(\frac{\frac{n^2}{4k^2}}{\ln(2^n)}\right)\\ &\geq \left(1 - \gamma - \frac{2\logb(2k)+\logb \ln(2) }{\logb(n)}\right)\logb(n) = (1 - o(1))\logb(n)
\end{align*} stages, where we use that $\gamma=o(1)$ and $k$ is subpolynomial in $n$.

It remains to show that if a graph $G$ contains an edge from every $F \in \cF$, then $G$ does not admit a proper $k$-coloring. Observe indeed that if there is a proper $k$-coloring of $G$, then, at least one color would be assigned to at least $\lceil \frac{n}{k} \rceil$ vertices. But $G$ contains an edge in every set of $\lceil \frac{n}{k} \rceil$ vertices, which is a contradiction. \hfill $\Box$

\subsection{Proof of the upper bound in Theorem~\ref{thm:kCol.multi} (Breaker's strategy)}

Breaker's strategy relies on multiple applications of the following lemma. 

\begin{lemma}[Corollary of Theorem 1.8 in~\cite{hefetz2007avoider}] \label{lem:acycle}
    Let $b \geq 1$ be an integer and $G$ be the union of at most $b+1$ edge-disjoint forests. Then Breaker wins the $(1:b)$ cycle game on $G$.
\end{lemma}

Breaker wants to force the board to be a forest, so that it is bipartite and thus $k$-colorable for each $k \ge 2$.
We explain below a strategy to achieve this in $\logb(n) + 1$ stages. 
In each stage we partition the board $\mathcal{X}_i = F_{i,1} \cup \dots \cup F_{i,k_i}$ into $k_i$ edge-disjoint forests, where $k_i$ is the smallest number such that such a partition exists. Let $\mathcal{F}_i$ be the collection of forests from such a partition. We show that Breaker can ensure that $k_{i+1} \leq \lceil \frac{k_i}{b+1} \rceil$, by using the following strategy. In Stage $i$, we split the board $\mathcal{X}_i$ into $\lceil \frac{k_i}{b+1} \rceil$ edge-disjoint boards $G_j$, such that each forest $F \in \mathcal{F}_i$ is contained in exactly one board $G_j$, and each board contains at most $b+1$ edge-disjoint forests $F \in \mathcal{F}_i$. Whenever Maker plays on some board $G_j$, Breaker plays on the same board according to the strategy given by Lemma~\ref{lem:acycle}. Thus, at the end of Stage $i$, Maker has claimed an acyclic graph on each board $G_j$. We conclude that there is a partition of $\mathcal{X}_{i+1}$ into at most $\lceil \frac{k_i}{b+1} \rceil$ edge-disjoint forests as wanted.

Using the fact that $\mathcal{X}_0=E(K_n)$ and thus $k_0 = \lceil \frac{n}{2} \rceil$, we conclude that $k_i \leq 1$ for $i \geq \logb(n) + 1$, and thus the board becomes a forest.\hfill $\Box$

\bigskip

%
%
%

\section{$H$-game}\label{sec:H-game.multi}

In this section we prove Theorem~\ref{thm:H-game.multi} and, again, we discuss Maker's and Breaker's strategy separately.

\subsection{Proof of the lower bound in Theorem~\ref{thm:H-game.multi} (Maker's strategy)}

Given any graph $H$, we choose
$$
\gamma := 2 \cdot \frac{m_2(H)(\ln(b) + \ln\logb(n) +5)}{\ln(n) - 2m_2(H)\ln\ln(n)}\, ,
$$
and assume $n$ to be large enough whenever needed.
We will prove 
$\tau(\mathcal{H}_{H,n},b)\geq \left(\frac{1}{m_2(H)} - o(1)\right) \log_{b+1}(n)$, as required.
We use the method of hypergraph containers, developed by Balogh, Morris, and Samotij~\cite{balogh2015independent}, and, independently, Saxton and Thomason~\cite{saxton2015hypergraph}, which has been already used in the context of $H$-games, for the first time in~\cite{nenadov2016threshold}. 
Before stating it, we introduce some notation and, given a set $S$, we define $\cT_{k, s}(S)$ as the following family of $k$-tuples of subsets of $S$,
    $$ \cT_{k, s}(S) := \left\{ (S_1, \ldots, S_k) \,\Big| \,\,  S_i \subseteq S \; \text{for} \; 1 \leq i \leq k \; \text{and} \; \Big|\bigcup_{i=1}^k S_i \Big| \leq s\right\}. $$

\begin{thm}[Theorem $2.3$ in~\cite{saxton2015hypergraph}] \label{thm:containers}
    For any graph $H$ there exist constants $n_0, s \in \mathbb{N}$ and $\delta \in (0,1)$ such that the following is true. For every $n\ge n_0$  there exist $t=t(n)$, pairwise distinct tuples $T_1,\ldots,T_{t} \in \cT_{s, sn^{2 - 1/m_2(H)}}(E(K_n))$ and sets
    $C_1,\ldots,C_{t} \subseteq E(K_n)$, such that
    \begin{enumerate}
        \item[(C1)] each $C_i$ contains at most $(1 - \delta)\binom{n}{2}$ edges,
        \item[(C2)] for every $H$-free graph $G$ on $n$ vertices there exists $1\le i\le t$ such that $T_i\subseteq E(G) \subseteq C_i$, where $T_i\subseteq E(G)$ means that all sets contained in $T_i$ are subsets of $E(G)$.
    \end{enumerate}
\end{thm}

Fix the constants $s,\delta$ from the above theorem with input $H$. The theorem states that there exists a not too large collection of containers $\{C_i: i\in [t(n)]\}$ of graphs, each of which contains not too many edges and with the property that every $H$-free graph is a subgraph of one such container $C_i$.
Hence, in order to occupy a copy of $H$, it suffices for Maker
to claim an edge in each complement of a container, i.e. $E(K_n)\setminus C_i$. We therefore aim to apply Lemma~\ref{lem:multistagediscrepancy} with $s=1$ and family
$$
\mathcal{F} := \left\{E(K_n)\setminus C_i: i\in [t(n)] \right\}\, .
$$

We let $k:=\min\{|F|:~F\in\mathcal{F}\}$ and we check condition~\eqref{property:b} of Lemma~\ref{lem:multistagediscrepancy}. Observe that
$\delta\binom{n}{2}\leq k \leq \binom{n}{2}$
and
$$|\mathcal{F}| \leq \binom{\binom{n}{2}}{sn^{2 - 1/m_2(H)}} \cdot (2^s)^{sn^{2 - 1/m_2(H)}} < n^{2sn^{2 - 1/m_2(H)}}\, .
$$
In particular,
$\frac{n^{1/m_2(H)}}{\ln^2 n} \leq \frac{k}{\ln (|\mathcal{F}|)} \leq n^2$ and therefore
\begin{align*}
\left( \frac{k}{\ln (|\mathcal{F}|)} \right)^{\gamma/2}
& \geq 
\left( \frac{n^{1/m_2(H)}}{\ln^2 n } \right)^{\gamma/2}
=
\exp\left[\frac{\gamma}{2} \left( \frac{1}{m_2(H)}\ln(n) - 2\ln\ln(n) \right)\right] \\
& =
\exp\left(5 + \ln(b) + \ln\logb(n)\right) > 40b\logb(n) > 20b\logb\left( \frac{k}{\ln (|\mathcal{F}|)} \right)\, . 
\end{align*}

Therefore Maker can claim an edge from each $F \in \cF$ for at least
\begin{align*}
    (1 - \gamma)\logb\left(\frac{k}{\ln(|\mathcal{F}|)}\right) &\geq (1 - \gamma)\logb\left(\frac{n^{1/m_2(H)}}{\ln^2 n}\right)\\ &\geq \left(\frac{1-\gamma}{m_2(H)} - \frac{2\logb \ln(n) }{\logb(n)}\right)\logb(n) = \left(\frac{1}{m_2(H)} - o(1)\right)\logb(n)
\end{align*} stages, where we use that $\gamma = o(1)$ since $b$ is subpolynomial in $n$. Threfore the lower bound of Theorem~\ref{thm:H-game.multi} is proven. \hfill $\Box$

\medskip

\subsection{Proof of the upper bound in Theorem~\ref{thm:H-game.multi} (Breaker's strategy)}

We split Breaker's strategy in two phases.
In the first phase, which will occupy the main part of the game, Breaker ensures that the board
will not have many copies of $H$ clustered together.
Then afterwards, in the second phase, Breaker can consider each cluster separately and
destroy all remaining copies of $H$ in a tiny number of stages.
To make this approach more precise, we first introduce the necessary concepts, which were also used in~\cite{milos_tibor}
for analysing games on random graphs.

\begin{dfn}[$K$-collection]
    Let $G$ and $K$ be graphs. We define the auxiliary graph $G_K$ to be the graph with vertices corresponding to the copies of $K$ in $G$, and two vertices being adjacent if the corresponding copies of $K$ have at least two vertices in common. 
    Let $\cK=\{K_1, \dots, K_s\}$ be the family of copies of $K$ in $G$ corresponding to a connected component of $G_K$. 
    Then the subgraph of $G$ induced by $\bigcup_{i \in [s]} K_i$ is called a $K$-collection, and we denote its vertex set and its number of vertices by $V(\cK)=\bigcup_{i \in [s]} V(K_i)$ and $v(\cK)=|V(\cK)|$, respectively.
\end{dfn}

\begin{dfn}[$s$-bunch]
    Let $(K_1, \dots, K_s)$ be a sequence of copies of $K$.
    Then $\bigcup_{i \in [s]} K_i$ is called an $s$-bunch if $V(K_i) \setminus \left(\bigcup_{j \in [i-1]} V(K_j)\right) \neq \emptyset$ and $\left|V(K_i) \cap \left(\bigcup_{j \in [i-1]} V(K_j)\right)\right| \ge 2$, for each $i = 2, \dots, s$.
\end{dfn}

It is easy to observe that every large enough collection contains a large bunch.

\begin{clm}
\label{clm:collection_to_bunch}
    Let $G$ be a graph and $t \in \mathbb{N}$. Then every $K$-collection $\cK$ of $G$ on at least $tv(K)$ vertices contains an $s$-bunch $B$ of copies of $K$ with $s \geq t$ and $tv(K) \le v(B) \le (t+1)v(K) $.
\end{clm}

\begin{proof}[Proof of Claim~\ref{clm:collection_to_bunch}]
    We start by taking any copy of $K$ in $\cK$ and then construct the bunch recursively as follows. 
    If $\bigcup_{i\in [m]} K_i$ is an $m$-bunch of copies of $K$, we select another copy $K_{m+1}$ of $K$ in the collection $\cK$, such that $V(K_{m+1}) \setminus \left(\bigcup_{j \in [m]} V(K_j)\right) \neq \emptyset$ and $\left|V(K_{m+1}) \cap \left(\bigcup_{j \in [m]} V(K_j)\right)\right| \ge 2$.
    Note that this will give an $(m+1)$-bunch of copies of $K$. 
    Since the family $\cK$ corresponds to a connected component of the auxiliary graph $G_K$, we are able to find such new copy of $K$ if  $V(\cK) \setminus \left(\bigcup_{i \in [m]} V(K_i) \right) \neq \emptyset$, i.e. until we cover all the vertices of $\cK$.
    In particular, since $v(\cK) \ge t v(K)$, we can construct an $s$-bunch $B=\bigcup_{i\in [s]} K_i$ of copies of $K$ with $tv(K)\leq v(B) <  (t+1) v(K) $.
    Moreover, since $ t v(K)  \leq v(B) \le v(K) + (s-1)(v(K)-2)$, we get $s \ge t$.
\end{proof}

Observe that for the $s$-bunch where any two copies of $K$ intersect in the same two adjacent vertices, we have $d(B) = \frac{ e(K) + (s-1) (e(K)-1))}{ v(K) + (s-1)(v(K)-2)}$, which tends to $\frac{e(K)-1}{v(K)-2}=m_2(K)$ as $s$ tends to infinity. 
Using a similar argument as in~\cite{milos_tibor}, we show that this is best possible in the following sense.

\begin{clm}
\label{clm:min 2-density}
    Let $K$ be a graph such that $d_2(K)=m_2(K)$, $s \in \mathbb{N}$ and $\delta>0$ such that for all $x\geq s-1$ we have $\frac{e(K)+m_2(K)x}{v(K)+x} \geq m_2(K)-\delta$.
    Then for any $s$-bunch $B$ of copies of $K$, we have $d(B) \ge m_2(K)-\delta$.
\end{clm}

\begin{proof}[Proof of Claim~\ref{clm:min 2-density}]
    Let $B=\bigcup_{i \in [s]} K_i$ be any $s$-bunch of copies of $K$.
    First observe that for every $S \subset K$ we have $\frac{e(S)-1}{v(S)-2}\leq \frac{e(K)-1}{v(K)-2}$, which can be rearranged as $\frac{e(K)-e(S)}{v(K)-v(S)} \geq \frac{e(K)-1}{v(K)-2} = m_2(K)$, which in turn gives
    \begin{equation}
    \label{eq:2-density}
        e(K)-e(S) \geq m_2(K) \cdot (v(K)-v(S))\, .
    \end{equation}
    Setting $S_i = K_i \cap (\bigcup_{j \in [i-1]} K_j)$ for each $i=2,\dots,s$, we have
    \begin{align*}
    d(B) = \frac{e(B)}{v(B)}
    	& = \frac{ e(K) + \sum_{i\geq 2}(e(K_i)-e(S_i)) }{ v(K) + \sum_{i\geq 2}(v(K_i)-v(S_i)) } \\
    	& \stackrel{\eqref{eq:2-density}}{\geq} \frac{ e(K) + m_2(K)\sum_{i\geq 2}(v(K_i)-v(S_i)) }{ v(K) + \sum_{i\geq 2}(v(K_i)-v(S_i))} 
    	  \geq m_2(K)-\delta\, ,
    \end{align*}
    where the last inequality follows from the assumption on $\delta$ and as $\sum_{i\geq 2}(v(K)-v(S_i)) \ge s-1$.
\end{proof}

Having the above claims in hand, we can now describe and analyse Breaker's strategy. Given any constant $\eps>0$ and any $n\in \mathbb{N}$ large enough, we show that Breaker can block all copies of $H$ in  
at most $\left(\frac{1}{m_2(H)} + \eps \right) \logb (n)$  stages.
For this,
let $K$ be any subgraph of $H$ such that $\frac{e(K)-1}{v(K)-2}=m_2(H)$, and notice that $d_2(K)=m_2(K)=m_2(H)$. 
Next, let $\delta=\delta(H,\eps)>0$ be a constant such that 
\[
\frac{1}{m_2(K)-\delta} < \frac{1}{m_2(K)}+\frac{\eps}{4}
\, ,
\]
and pick $t=t(H,\eps,\delta)\in \mathbb{N}$ such that for all $x\geq t-1$ it holds that
\[
\frac{e(K)+m_2(K)x}{v(K)+x} \geq m_2(K)-\delta
~~~
\text{and} 
~~~
\frac{(t+2)v(K)}{(m_2(K)-\delta) \cdot t \cdot v(K)-1} < \frac{1}{m_2(K)-\delta}+\frac{\eps}{4}\, .
\]

As already pointed out, Breaker's strategy is based on two phases: first he blocks all $K$-collections on at least $tv(K)$ vertices (see Claim~\ref{clm:ES}), and then he blocks all remaining copies of $K$.
Since $K$ is a subgraph of $H$, at this point Breaker will have blocked all copies of $H$ as well.

\begin{clm}[First phase of Breaker's strategy]
\label{clm:ES}
    Breaker has a strategy so that after $\left( \frac{1}{m_2(K)}+ \frac{\eps}{2} \right) \logb (n)$ stages, the board
    does not contain a $K$-collection with at least 
    $tv(K)$ vertices.
\end{clm}

\begin{proof}[Proof of Claim~\ref{clm:ES}]
    By Claim~\ref{clm:collection_to_bunch} we know that,
	if the board contains a $K$-collection on at least $tv(K)$ vertices, then it must also contain an $s$-bunch $B$ with $s\geq t$ and $tv(K)\leq v(B)\leq (t+1)v(K)$. Hence, Breaker can concentrate on blocking such bunches:
    \[
    \cF :=
    \left\{
    B=\bigcup_{i\in [s]} K_i:~ 
    \begin{array}{c}
    B~\text{ is an $s$-bunch of copies $K_i$ of $K$ in $K_n$ with }\\
    s\geq t\text{ and } tv(K)\leq v(B) \leq (t+1)v(K)
    \end{array}
    \right\}\, .
    \]
Let $\mathcal{F}_0:=\mathcal{F}$, and throughout the game denote with $\mathcal{F}_i\subseteq \mathcal{F}_{i-1}$ the family of all elements $F\in\mathcal{F}_{i-1}$ that Maker has fully occupied at the end of Stage $i$. In order to prove the claim, we must show that Breaker has a strategy to ensure $\mathcal{F}_k=\varnothing$ for $k=\left( \frac{1}{m_2(K)}+ \frac{\eps}{2} \right) \logb (n)$.
\smallskip    
    
    Now, as any bunch in $\cF_0$ has at most $(t+1)v(K)$ vertices, we have
    \begin{equation}
    \label{eq:F0}
       |\cF_0| \le n^{(t+1)v(K)} \cdot 2^{((t+1)v(K))^2} < n^{(t+2)v(K)}\, .
    \end{equation} 
    Using Theorem~\ref{beck_variant}, Breaker has a strategy to ensure that 
    \begin{equation}
    \label{eq:ES}
        |\mathcal{F}_i| \leq \sum_{B\in \mathcal{F}_{i-1}} (b+1)^{-e(B)+1} 
          \leq |\mathcal{F}_{i-1}|\cdot (b+1)^{-(m_2(K)-\delta) \cdot t \cdot v(K)+1}\,
    \end{equation}
    for each positive $i \in \mathbb{N}$, where in the last inequality we use that $e(B)=d(B)\cdot v(B) \geq (m_2(K)-\delta)\cdot t \cdot v(K)$ for all $B\in \mathcal{F}_{i-1}\subseteq \mathcal{F}_0$, which follows from the Claim~\ref{clm:min 2-density} (note its assumptions hold by the choice of $t$).
    Combining~\eqref{eq:F0} and~\eqref{eq:ES}, we observe that
    \[
    |\mathcal{F}_k| < n^{(t+2)v(K)}\cdot (b+1)^{k \cdot [-(m_2(K)-\delta) \cdot t \cdot v(K)+1]} \leq 1 \]
since, by our choice of $\delta$ and $t$, we have $k > \tfrac{(t+2)v(K)}{(m_2(K)-\delta)\cdot t \cdot v(K)-1} \logb (n)$.
Therefore $\mathcal{F}_k=\varnothing$ and this finishes the proof of the claim.
\end{proof}

Now, consider the first moment when Breaker made sure that every remaining $K$-collection has fewer than $tv(K)$ vertices, and denote them by $\mathcal{K}_1,\ldots,\mathcal{K}_{\ell}$, $\ell\in \mathbb{N}_0$.
Since any two such collections are edge-disjoint by definition, we know then that each remaining copy of $K$ must appear in a unique collection.
From now on, in each further stage, Breaker plays as follows: whenever Maker claims an edge of $E(\mathcal{K}_i)$ for some $i\in [\ell]$, Breaker claims as many edges as possible of the same collection. In all other cases, Breaker plays arbitrarily. 
Since each of the collections $\mathcal{K}_i$ has fewer than 
$(tv(K))^2$ edges, it takes less than $\logb ((tv(K))^2)$ stages
until from each collection there is at most one edge left and hence all copies of $K$ are blocked. 
Combining the two phases, Breaker wins within
\[
\left( \frac{1}{m_2(K)}+ \frac{\eps}{2} \right) \logb (n) + 2\logb (tv(K)) \le \left( \frac{1}{m_2(K)}+ \eps \right) \logb (n) = \left( \frac{1}{m_2(H)}+ \eps \right) \logb (n)
\]
stages. Hence, the upper bound of Theorem~\ref{thm:H-game.multi} is proven. \hfill $\Box$

\bigskip

%
%
%

\section{Pancyclicity game}\label{sec:PAN.multi}

In this section we prove Theorem~\ref{thm:PAN.multi}.
The upper bound $\tau(\mathcal{PAN}_{n},b) \leq \left( \frac{1}{2} + o(1) \right) \log_{b+1}(n)$ follows from Theorem~\ref{thm:H-game.multi}. 
Indeed Breaker plays according to his strategy given by Theorem~\ref{thm:H-game.multi} with $H = C_3$. By doing this, he can ensure that the board does not contain any triangle after $\left(\frac{1}{2} + o(1)\right)\logb(n)$ stages and thus it cannot be pancyclic.
For the lower bound, we will use the following criterion for a graph to be pancyclic, which is a corollary of Theorem~1.1 in~\cite{keevash2010pancyclicity}.

\begin{thm}[Corollary from Theorem 1.1 in \cite{keevash2010pancyclicity}]\label{thm:pancriterion}
    Let $G$ be a graph on $n$ vertices such that
    \begin{enumerate}[label=\upshape(P\arabic*)]
        \item \label{pancyclic-independent} every independent set of $G$ is of size at most $\sqrt{n}$;
        \item \label{pancyclic-connected} $G$ is $600\sqrt{n}$-vertex-connected,
    \end{enumerate}
    then $G$ is pancyclic.
\end{thm}

\subsection{Proof of the lower bound in Theorem~\ref{thm:PAN.multi} (Maker's strategy)}

Let $n$ be large enough and consider the family $\mathcal{F} := \mathcal{F}_1 \cup \mathcal{F}_2 \cup \mathcal{F}_3$ with
\begin{align*}
    \mathcal{F}_1 &:= \{E_{K_n}(A): A \subset V(K_n), |A| = \sqrt{n}\},\\
    \mathcal{F}_2 &:= \{E_{K_n}(A,B):  A,B \subset V(K_n), A \cap B = \emptyset, |A| = 1, |B| = n - 700\sqrt{n}\},\\
    \mathcal{F}_3 &:= \{E_{K_n}(A,B): A,B \subset V(K_n), A \cap B = \emptyset, |A| = \sqrt{n}, |B| = \sqrt{n}\},
\end{align*}
and set $k_j := \min\{|F|: F \in \mathcal{F}_j\}$ for $j \in [3]$. Maker plays according to Lemma~\ref{lem:multistagediscrepancy} with $$\gamma = 2 \cdot \frac{\ln(b) + \ln \logb(n) + 5}{\ln(\sqrt{n}) - \ln \ln(n) - \ln(3000)}\, ,$$ $s=3$ and hypergraph $\mathcal{H} = (E(K_n), \mathcal{F})$. Note that $\gamma = o(1)$ since $b$ is subpolynomial in $n$.

First, we check whether the condition~\eqref{property:b} of Lemma~\ref{lem:multistagediscrepancy} is met. We have $|\mathcal{F}_1| = \binom{n}{\sqrt{n}}$, $|\mathcal{F}_2| = n \binom{n-1}{700\sqrt{n}-1}$, and $|\mathcal{F}_3| = \binom{n}{\sqrt{n}}\binom{n-\sqrt{n}}{\sqrt{n}}$. Further, we have $k_1 = \binom{\sqrt{n}}{2}$, $k_2 = n - 700\sqrt{n}$, and $k_3 = n$.
Note that $\frac{n}{4} \leq k_j \leq n$ and $\sqrt{n}^{\sqrt{n}} \leq |F_j| \leq n^{700\sqrt{n}}$ and thus $\frac{\sqrt{n}}{3000\ln(n)} \leq \frac{k_j}{\ln(3|\mathcal{F}_j|)} \leq 2\frac{\sqrt{n}}{\ln(n)}$ for every $j \in [3]$. Using this, we can estimate for each $j \in [3]$, 
\begin{align*}
    \left(\frac{k_j}{\ln(3|\mathcal{F}_j|)}\right)^{\frac{\gamma}{2}} & \geq \left(\frac{\sqrt{n}}{3000\ln(n)}\right)^{\frac{\gamma}{2}} = \exp\left(\ln(b) + \ln(\logb(n)) + 5\right)\\ 
    &> 20b\logb(n) \geq 20b\max\left\{1,\logb\left(\frac{k_j}{\ln(3|\mathcal{F}_j|)}\right)\right\},
\end{align*}

and thus~\eqref{property:b} holds. We conclude that Maker can claim an edge of each $F \in \mathcal{F}$ for at least \[(1 - \gamma)\logb\left(\frac{\sqrt{n}}{3000\ln(n)}\right) \geq \left(\frac{1}{2} - \gamma - \frac{\logb(3000\ln(n))}{\logb(n)}\right)\logb(n) = \left(\frac{1}{2} - o(1)\right)\logb(n)\] stages.

It remains to show that if a graph $G$ contains an edge from every $F \in \cF$, then $G$ fulfills \ref{pancyclic-independent} and \ref{pancyclic-connected}. Let $A \subset V(K_n)$ be any vertex set of size $|A| = \sqrt{n}$. Then, using the subfamily $\cF_1$, the graph $G$ has at least one edge within $A$. Therefore, every independent set can be of size at most $\sqrt{n}$ as required by \ref{pancyclic-independent}. Further, let $V(K_n) = A \cup B \cup C$ be any partition of $V(K_n)$ with $|C| = 600\sqrt{n}$. We show that $G$ contains an edge between $A$ and $B$, and thus it is $600\sqrt{n}$-vertex-connected as required by \ref{pancyclic-connected}. Assume $|A| \geq \sqrt{n}$ and $|B| \geq \sqrt{n}$. In this case, using the subfamily $\cF_3$, the graph $G$ contains an edge between $A$ and $B$. Assume otherwise without loss of generality that $|A| < \sqrt{n}$. Using $\mathcal{F}_2$ instead, every vertex in $A$ has at least $700\sqrt{n} > |A| + |C|$ neighbours in $G$, so there needs to be a neighbour in $B$.\hfill $\Box$

\bigskip

%
%
%
%

\section{Connectivity game}\label{sec:Conn.multi_unbiased}

Theorem~\ref{thm:Conn.multi_unbiased} has already been proven in the Bachelor thesis~\cite{barkey} of the first author and we include the argument here for completeness. It is an easy application of Lehman's Theorem~\cite{lehman}
 of which we use the following formulation that can be found in e.g.~\cite{hefetz2014positional}.

\begin{thm}[Theorem 1.1.3 in~\cite{hefetz2014positional}]\label{thm:lehman}
Let $G=(V,E)$ be a
graph on $n$ vertices which admits two edge-disjoint spanning trees. Then in the unbiased Maker-Breaker game on the edge set of $G$, Maker, even as a second player, has a strategy to build a connected spanning tree of $G$ within $n-1$ moves.
\end{thm}

\begin{proof}[Proof of Theorem~\ref{thm:Conn.multi_unbiased}]
After more than $\log_2(n)-1$ rounds, the board has fewer than $n-1$ edges and cannot contain a spanning tree. Hence, $\tau(\mathcal{C}_n,1)\leq \lfloor \log_2(n)-1\rfloor$ follows.

For the lower bound, Maker's goal is to ensure that Stage $i$ is played on a board which contains $\lfloor \frac{n}{2^i} \rfloor$ edge-disjoint spanning trees, for every $i\leq \lfloor \log_2(n)-1\rfloor$. As $\lfloor \frac{n}{2^t} \rfloor \geq 2$ for $t=\lfloor \log_2(n)-1\rfloor$, Theorem~\ref{thm:lehman} then guarantees that Maker can still claim a spanning tree in round $t$, which yields
$\tau(\mathcal{C}_n,1)\geq \lfloor \log_2(n)-1\rfloor$.

In order to achieve Maker's goal, notice first that it is a well known fact that $E(K_n)$ contains $\lfloor \frac{n}{2} \rfloor$ edge-disjoint spanning trees. Hence, the case $i=1$ of Maker's goal is obvious, and we can proceed by induction. Assume the board at Stage $i-1$ has $\lfloor \frac{n}{2^{i-1}} \rfloor$ spanning trees. Then Maker uses these spanning trees to form $\lfloor \frac{n}{2^i}\rfloor$ pairwise disjoint pairs of spanning trees (possibly ignoring one further spanning tree), and she plays on each pair separately using the strategy from Theorem~\ref{thm:lehman}, i.e. she always plays on the same pair that Breaker played on in his previous move.
In this way Maker occupies $\lfloor \frac{n}{2^i}\rfloor$ edge-disjoint spanning trees, which then belong to the board on which the next stage is played, as wanted.
\end{proof}

\bigskip

\section{Concluding remarks}\label{sec:concluding}

In this paper we have introduced multistage Maker-Breaker games and determined the duration of such games for several natural graph properties. It would be interesting to continue this line of research and hence we suggest the following problem.

\begin{prob}
Determine $\tau(\mathcal{H},b)$ when $b\in\mathbb{N}$ and $\mathcal{H}$ is the hypergraph on vertex set $E(K_n)$, with hyperedges being the edge sets of (i) triangle factors, (ii) copies of a fixed spanning tree or (iii) powers of Hamilton cycles.
\end{prob}

In particular, it would be interesting to determine whether the above games exhibit some random graph intuition.
We have shown that for both the connectivity game
and the Hamilton cycle game, a random graph intuition holds at least asymptotically. However, it seems challenging to get the second order terms right or to see whether Maker can always do at least as good as the random graph argument suggests.

\begin{ques}
Given $b\in\mathbb{N}$ subpolynomial in $n\in\mathbb{N}$. Is it true that
$$\tau(\mathcal{HAM}_n,b)\geq \logb(n) - (1-o(1))\log_{b+1}(\ln (n))\, ?$$
\end{ques}

As already pointed out in the introduction, in all our results it happens that $\tau(\mathcal{H},b)$ is asymptotically the same as $\logb(b_{\mathcal{H}})$, where $b_{\mathcal{H}}$ denotes the threshold bias for the unbiased Maker-Breaker game on the hypergraph $\mathcal{H}$. We wonder whether this is always the case.

\begin{ques}
Does there exist a hypergraph $\mathcal{H}=(\mathcal{X},\mathcal{F})$ and a positive integer $b$, for which
$\tau(\mathcal{H},b)$ and $\log_{b+1}(b_{\mathcal{H}})$ are not asymptotically the same?
\end{ques}

As a natural next step one may also consider multistage variants of other positional games. Such variants of Waiter-Client games and Client-Waiter game on $K_n$, and of Maker-Breaker games on the random graph $G_{n,p}$ are already work in progress. But even other kinds of positional games or boards would be of interest.

\smallskip

Moreover, the following variant of multistage games springs to mind, which we may call \textit{multistage game with stop}. Let a hypergraph $\mathcal{H}=(\mathcal{X},\mathcal{F})$ and a bias $b$ be given, and define $\mathcal{X}_0:=\mathcal{X}$.
For $i \ge 1$, Stage $i$ is played on the board $\mathcal{X}_{i-1}\subseteq \mathcal{X}$, and it ends the first time Maker claims a winning set from $\mathcal{F}_{i-1}\subseteq \mathcal{F}$ completely. Then we define the next board $\mathcal{X}_i$ to consist of all the elements which have been claimed by Maker or are free by the end of Stage $i$, and we let $\mathcal{F}_i\subseteq \mathcal{F}_{i-1}$ be the family of those winning sets which are still fully contained in $\mathcal{X}_i$. Similarly to the threshold $\tau(\cH,b)$, we may define the duration of this game, and denote it by $\tau^{\text{stop}}(\mathcal{H},b)$.

It is easy to see that $\tau^{\text{stop}}(\mathcal{H},b)\geq \tau(\mathcal{H},b)$ always holds.
Moreover, for the connectivity game and hence also the Hamilton cycle game, we obtain that this bound is asymptotically tight, as
$\tau^{\text{stop}}(\mathcal{C}_n,b)\leq (1+o(1))\log_{b+1}(n)$ can be shown if $b$ is subpolynomial in $n$. Indeed Breaker just needs to isolate a single vertex.

However, if we consider local properties instead, e.g.\ the $H$-game, there can be a huge difference between 
$\tau^{\text{stop}}(\mathcal{H},b)$ and
$\tau(\mathcal{H},b)$. Indeed, using that on dense graphs Maker can claim a copy of $H$ fast and applying Tur\'an's Theorem, it is straightforward to prove that 
$\tau^{\text{stop}}(\mathcal{H}_{H,n},1)=\Theta(n^2)$. It would be interesting to understand this variant much better. Hence, we suggest the following problem.

\begin{prob}
Given any graph $H$ and any constant bias $b$, determine $c=c(H,b)$ such that\break $\tau^{\text{stop}}(\mathcal{H}_{H,n},b)=(c\pm o(1))n^2$. 
\end{prob}

Already the case when $H=K_3$ and $b=1$ is open and would be of interest.

\medskip

{\bf Acknowledgment.} We would like to thank the organizers of the online workshop \emph{Positional games on sparse/random graphs} of the Sparse Graphs Coalition, where part of this work has been done.
The first author would also like to thank the group of Prof.~Anusch Taraz for many fruitful discussions while working on his Bachelor thesis.

\bigskip

\bibliographystyle{amsplain}
\bibliography{references}

\end{document}